\documentclass[12pt,english]{amsart}
\usepackage[letterpaper]{geometry}
\usepackage[utf8]{inputenc}
\usepackage{amsmath}
\usepackage{amsthm}
\usepackage{amssymb}
\usepackage{hyperref}
\usepackage[normalem]{ulem}

\usepackage{babel,blindtext}
\usepackage{graphicx}
\newtheorem{theorem}{Theorem}[section]
\newtheorem{lemma}[theorem]{Lemma}
\newtheorem{proposition}[theorem]{Proposition}
\newtheorem{definition}[theorem]{Definition}
\newtheorem{corollary}[theorem]{Corollary}
\newtheorem{remark}[theorem]{Remark}

\begin{document}

    \title[Conditions for stable equilibrium in a Cournot model]{Conditions for stable equilibrium in Cournot duopoly models with tax evasion and time delay}

\author{Ra\'ul Villafuerte-Segura}
\address{Centro de Investigaci\'on en Tecnolog\'ias de  Informaci\'on y Sistemas, Universidad Aut\'onoma del Estado de Hidalgo, Pachuca, Hidalgo, {\sc M\'exico}, 42184}
\email{villafuerte@uaeh.edu.mx}

\author{Eduardo Alvarado-Santos}
\email{mslalo\_58@hotmail.com }

\author{Benjam\'in A.~Itz\'a-Ortiz}
\address{Centro de Investigaci\'on en Matem\'aticas, Universidad Aut\'onoma del Estado de Hidalgo, Pachuca, Hidalgo, {\sc{México}}, 42184}
\email{itza@uaeh.edu.mx}


\date{August 22, 2019}

\dedicatory{}


\begin{abstract}
We provide conditions for stable equilibrium in Cournot duopoly models with tax evasion and time delay. We prove that our conditions actually imply asymptotically stable equilibrium and delay independence.  Conditions include the same marginal cost and equal probability for evading taxes. We give examples of cost and inverse demand functions satisfying the proposed conditions. Some economic interpretations of our results are also included.
\end{abstract}

\maketitle

\section*{Introduction}
Since its introduction in 1838, the Cournot  model  \cite{C} has provided abundant cases  of study for both basic and applied research. As time delay has become an inherent property  often needed when modeling natural phenomena, including  economic dynamics, the importance of discussing a Cournot model with time delay was recognized decades ago  \cite{HR}. Recently, it has  become  an active research area
\cite{EM,GGS,GMS,MS,MSY,O}.
In \cite{GR}, a Cournot oligopoly model with tax evasion was introduced, and subsequently studied, in the duopoly setting, by other authors {\cite{BNO,IM,N}}, where in addition, the introduction of a {time delay}  was made.The time delay in the Cournot duopoly model is justified just 
by considering  that there is a first a firm in the duopoly which enters the market  followed by the second firm entering the market some time later.
Besides the two classical variables in the model representing the supplies of the two competitors, the modeling of tax evasion has introduced two new variables representing the declared revenue upon which tax due is calculated. This increment in the number variables in the model establishes new mathematical challenges. 
The aim of this paper is to provide an abstract framework in which it is convenient to establish general conditions for stable equilibrium of a Cournot duopoly model with tax evasion and time delay. The conditions provided in this paper turn out to be so robust that in fact we are able to prove not only stability of the equilibrium but asymptotic stability and independence of the time delay as well. A somewhat similar result was obtained in \cite[Theorem~4]{GMS} for a particular duopoly Cournot model with two time delays but without the tax evasion.

The tools used in this paper are based on standard mathematical economics conditions such as strictly decreasing costs and inverse demand functions. What we realized was that, with a few extra mathematical requirements, such as that both competitors have the same first and second derivatives of their cost functions at the equilibrium point, together with a couple of inequalities involving second derivatives of the functions involved in the model, we were able to acquire sufficient conditions for asymptotic stability of the equilibrium and and that such stability was independent of time delay.  To illustrate our results, we also present some standard examples of cost and inverse demand functions along with the constraints needed to fulfill our conditions. With these examples, we expect to show that our conditions, though robust, are quite achievable.

We divide this work in four sections. In Section 1, we will review the Cournot model studied in the paper along some basic results for their use in later sections. We will give the announced conditions for stability and prove our main results in Section 2. In Section~3, we will show examples of cost and inverse demand functions along some constraints to meet the conditions for stability presented in the paper. Finally, we will give conclusions and some interpretations of our results in Section~4.

\section{The model}
In this section we will define and explain the Cournot model studied in this paper and will also  derive some elementary results needed for later. This Cournot model was originally introduced in \cite{GR}. We begin by letting the variable $x_1\geq 0$ and $x_2\geq 0$ denote the quantities produced by each firm in the duopoly, while $z_1\geq 0$ and $z_2\geq 0$ will denote the income each firm declares as revenue to the tax authority. The tax rate is denoted by $0<\sigma<1$ and the probability that the $i$-firm is caught evading taxes is $0< q_i< 1$.  The functions involved in the definition  of the Cournot model are assumed to be twice differentiable with continuous second derivatives. These functions and their standard  mathematical economics properties  are resumed in Table~\ref{table}.

\begin{center}
\begin{table}[h]
\begin{tabular}{|c|c| c|c|} \hline
   \bf Function & \bf Notation & \bf Property  & \bf Inequality \\ \hline
    Inverse demand & $p$ & 
     positive & $p(u)>0$\\ \cline{3-4}
      && strictly decreasing & 
        $p\sp\prime (u)<0$\\ \hline 
        Cost  & $C_i$ &
        nonnegative & $C_i(u)\geq 0$ \\ \cline{3-4}
        && strictly increasing
        & $C_i\sp\prime(u)>0$\\ \hline 
    Fine &   $F$ & nonnegative
    & $F(u)\geq 0$ \\\cline{3-4} & & strictly increasing &
        $F\sp\prime(u) >0$  \\ \cline{3-4}
        \hline
\end{tabular}
\caption{\label{table}The functions of the Cournot model and their properties}
\end{table}
\end{center}


For $i=1,2$, we define the profit function $P_i$ as the sum of two functions.   The first addend is
\[
\left(1 - q_i\right) \big(x_{i}\, p (x_1+x_2) - C_{i} (x_{i}) - \sigma z_{i}\bigr),
\]
which represent the probability $1-q_i$ of not being caught evading tax times the result of subtracting to the income actual $x_ip(x_1+x_2)-C_i(x_i)$ the tax paid on the declared income $z_i$. The second addend in the profit function $P_i$ is
\[
q_i \, \biggl(\left(1 - \sigma\right) x_{i}\, p (x_1+x_2) - C_{i} (x_{i}) - F \left(x_{i}\, p (x_1+x_2) - z_{i}\right)\biggr), 
\]
which represent the probability $q_i$ of being caught evading tax times the result of subtracting to the income $x_ip(x_1+x_2)-C_i(x_i)$ the tax due $\sigma x_ip(x_1+x_2)$ and a penalty $F(x_ip(x_1+x_2)-z_i)$ on the evaded amount.
Hence, we may rewrite the profit  function $P_i$ corresponding to the $i$-th firm as 
\begin{align}\label{FuncionBeneficio}
P_i=P_{i} \left(x_1,x_2,z_1,z_2\right)  
 & =   \notag \left( 1-q_i\, \sigma\right) x_i\, p\left( x_1 +x_2 \right)- C_i(x_i)   \\ 
 &\qquad -\left(1-q_i\right)\sigma z_i-q_i\, F\left( x_i\, p\left( x_1+x_2 \right) -z_i \right) .
\end{align}
We assume that the $i$-th firm only have the variables $x_i$ and $z_i$ under its control, so that $x_j$ and $z_j$ for $j\not=i$ are regarded as given. Therefore, to maximize profits, we will regard each $P_i$ as function on the two variables $x_i$ and $z_i$. Conditions for maximizing the profit functions $P_1$ and $P_2$ are given in the following.

\begin{proposition}\label{maximum}
  Suppose that $x_i=x_i\sp\ast$ and $z_i=z_i\sp\ast$, for  $i=1,2$, satisfy the following system of four equations
  \begin{align}\label{Cournot}
 \bigl(1-  \sigma \bigr) \frac{\partial}{\partial x_i}\bigl( x_i p \left(x_{1} + x_{2}\right)\bigr)   &= C_i\sp\prime (x_{i})  ,
 \notag \\
  q_iF'\left(x_{i}p (x_{1} + x_{2}) - z_{i}\right) &= \left( 1-q_i\right)\sigma.
\end{align}
For $j\not= i$, fix the values $x_j=x_j\sp\ast$ and $z_j=z_j\sp\ast$ in $P_i$ so that $P_i$ becomes a function on the two variables $(x_i,z_i)$. Then the profit function $P_i$ reaches a local maximum at $x_i=x_i\sp\ast$ and $z_i=z_i\sp\ast$ if 
\begin{equation*}
F\sp{\prime\prime}(x_i\sp\ast p(x_1\sp\ast +x_2\sp\ast)-z_i\sp\ast)>0 
\hspace*{0.5cm} \text{ and }\hspace*{0.5cm} \dfrac{\partial^2}{\partial x_i^2}\Bigr|_{\substack{x_i=x_i\sp\ast\\z_i=z_i\sp\ast}} \left(x_ip(x_i+x_j\sp\ast)-\frac{C_i(x_i)}{1-\sigma}\right)<0
\end{equation*}.
\end{proposition}
\begin{proof}
If we fix the values $x_2=x_2\sp\ast$ and $z_2=z_2\sp\ast$ in $P_1$ so that $P_1$ is a function on the two variables $x_1$ and $z_1$ 
then we first  check that $(x_1\sp\ast,z_1\sp\ast)$ is a critical point of $P_1$. For this purpose, we solve the equation $\nabla P_1(x_1,z_1)=0$. We obtain
\begin{align*}
    \dfrac{\partial P_1}{\partial x_1}&= \bigl(1- q_1 \sigma- q_1  F\sp\prime \bigl(x_{1} p \left(x_{1} + x_{2}\sp\ast\right) - z_{1}\bigr)\bigr) \frac{\partial}{\partial x_1}\bigl( x_1 p \left(x_{1} + x_{2}\sp\ast\right)\bigr)  - C'_{1} (x_{1}) =0,\\
     \dfrac{\partial P_1}{\partial z_1}&= -(1-q_1)\sigma +  q_1\, F'\left(x_{1}p (x_{1} + x_{2}\sp\ast) - z_{1}\right) =0.
\end{align*}
From the second equation above we obtain $\sigma= q_1 \sigma+ q_1  F\sp\prime \left(x_{1}\, p \left(x_{1} + x_{2}\sp\ast\right) - z_{1}\right)$ and substituting this $\sigma$ in the first equation we apply the hypothesis (\ref{Cournot}) to conclude that a critical point of $P_i$ is indeed achieved at $x_1=x_1\sp\ast$ and $z_1=z_1\sp\ast$, as wanted. 
To verify that it is  a  relative maximum for $P_i$, by the second partial derivative test, we require two conditions. The first condition is

\begin{equation*}
 -q_1\,F\sp{\prime\prime}              \left(x_1\sp\ast p(x_1\sp\ast+x_2\sp\ast)-z_1\sp\ast\right)=\dfrac{\partial^2 P_1}{\partial z_1^2}\Bigr|_{\substack{x_1=x_1\sp\ast\\z_1=z_1\sp\ast}}<0.
 \end{equation*}
So that this first condition holds   if and only if $F\sp{\prime\prime}(x_1\sp\ast p(x_1\sp\ast+x_2\sp\ast)-z_1\sp\ast)>0$. On the other hand, the second condition is

\begin{align*}
0&<\dfrac{\partial^2 P_1}{\partial x_1^2}\Bigr|_{\substack{x_1=x_1\sp\ast\\z_1=z_1\sp\ast}}\dfrac{\partial^2 P_1}{\partial z_1^2}\Bigr|_{\substack{x_1=x_1\sp\ast\\z_1=z_1\sp\ast}} -
\biggl(  \dfrac{\partial^2 P_1}{\partial x_1\partial z_1}\Bigr|_{\substack{x_1=x_1\sp\ast\\z_1=z_1\sp\ast}} \biggr)^2 \\
& = 
\biggl( (1-\sigma)\dfrac{\partial^2}{\partial x_1^2}\Bigr|_{\substack{x_1=x_1\sp\ast\\z_1=z_1\sp\ast}}\left(x_1p(x_1+x_2\sp\ast)-\frac{C_1(x_1)}{1-\sigma}\right) \\
&\qquad
 -q_1 F\sp{\prime\prime} \left( x_1\sp\ast p\left(x_1\sp\ast+x_2\sp\ast\right)-z_1\sp\ast \right)\biggl(\frac{\partial}{\partial x_1}\Bigr|_{\substack{x_1=x_1\sp\ast\\z_1=z_1\sp\ast}}x_1p\left(x_1+x_2\sp\ast\right)
 \biggr)^2 \biggr)
\, \biggl( -q_1\,F\sp{\prime\prime}              \left(x_1\sp\ast p(x_1\sp\ast +x_2\sp\ast)-z_1\sp\ast\right) \biggr)\\
&\quad -\biggl( q_1 F\sp{\prime\prime} \left(  x_1\sp\ast p(x_1\sp\ast+x_2\sp\ast)-z_1\sp\ast \right)\frac{\partial}{\partial x_1}\Bigr|_{\substack{x_1=x_1\sp\ast\\z_1=z_1\sp\ast}}x_1p\left(x_1+x_2\sp\ast\right) \biggr)^2\\
&=- q_1 (1-\sigma) F\sp{\prime\prime} \left(  x_1\sp\ast p(x_1\sp\ast+x_2\sp\ast)-z_1\sp\ast \right)\dfrac{\partial^2}{\partial x_1^2}\Bigr|_{\substack{x_1=x_1\sp\ast\\z_1=z_1\sp\ast}}\left(x_1p(x_1+x_2\sp\ast)-\frac{C_1(x_1)}{1-\sigma}\right). \\
\end{align*}
 Since the second derivative of $F$ is positive by hypothesis, the second condition  holds if and only if $\dfrac{\partial^2}{\partial x_1^2}\Bigr|_{\substack{x_1=x_1\sp\ast\\z_1=z_1\sp\ast}}\left(x_1p(x_1+x_2\sp\ast)-\frac{C_1(x_1)}{1-\sigma}\right)<0$.  This proves the maximality assertion for $P_1$. The proof of the assertion for $P_2$ is analogous.
\end{proof}

For our study we will request the following two inequalities, for each $i=1,2$. All functions are assumed to be evaluated at the points $x_i={x_i}\sp\ast$ and $z_i={z_i}\sp\ast$ which are the solution of the system of equations~(\ref{Cournot}).

\begin{align}
   \label{In2}   
   \frac{\partial^2P_i}{\partial x_i^2}\frac{\partial^2 P_i}{\partial z_i^2}- \left(\frac{\partial^2 P_i}{\partial x_i\partial z_i}\right)^2
    &> \left|\frac{\partial^2P_i}{\partial x_1\partial x_2}\frac{\partial^2 P_i}{\partial z_i^2}-\frac{\partial^2 P_i}{\partial x_1\partial z_i}
    \frac{\partial^2 P_i}{\partial x_2\partial z_i}\right| \\ \label{In3} 
    -\frac{\partial^2 P_i}{\partial x_i^2}&\geq \left| \frac{\partial^2 P_i}{\partial x_1\partial x_2} \right|
\end{align}

The inequalities (\ref{In2}) and (\ref{In3}) in fact imply $\dfrac{\partial^2 P_i}{\partial z_i^2}<0$ and so, by Proposition~\ref{maximum}, they
are sufficient conditions for  $P_i$ to reach a maximum at $x_i={x_i}\sp\ast$ and $z_i={z_i}\sp\ast$.  On the other hand, they turn out to be useful in our proofs of stability and, fortunately, they turn out to be satisfied by many standard examples, as it will be shown in Section~\ref{examples}. In the following two lemmas we will establish sufficient conditions for inequalities (\ref{In2}) and (\ref{In3}) to hold.

\begin{lemma}\label{lemma1}
If $\dfrac{\partial^2 x_ip(x_1+x_2)}{\partial x_1 \partial x_2}\leq 0$ and $C_i\sp{\prime\prime}(x_i)\geq 0$
then 
\begin{equation}\label{Inxi2>x1x2}
- \dfrac{\partial^2}{\partial x_i^2}\left(x_ip(x_1+x_2)-\frac{C_i(x_i)}{1-\sigma}\right)>\left|\frac{\partial^2}{\partial x_1\partial x_2} x_ip(x_1+x_2) \right|.
\end{equation}
If in addition we assume $F\sp{\prime\prime}(x_i p(x_1+x_2)-z_i)>0$ then  inequalities~(\ref{In2}) and (\ref{Inxi2>x1x2}) are equivalent.

\end{lemma}
\begin{proof}
We have $\dfrac{\partial^2x_ip(x_1+x_2)}{\partial x_i^2}=\dfrac{\partial^2x_ip(x_1+x_2)}{\partial x_1\partial x_2}+p\sp\prime(x_1+x_2)$, and by assumption $p\sp\prime(x_1+x_2)<0$. Therefore
\begin{equation}\label{sum1}
-\dfrac{\partial^2x_ip(x_1+x_2)}{\partial x_i^2}>-\dfrac{\partial^2x_ip(x_1+x_2)}{\partial x_1\partial x_2}=\left|\dfrac{\partial^2x_ip(x_1+x_2)}{\partial x_1\partial x_2}\right|.
\end{equation}
On the other hand, since
$1-\sigma>0$ and by hypothesis $C\sp{\prime\prime}(x_i)\geq 0$,  we get $\dfrac{C\sp{\prime\prime}(x_i)}{1-\sigma}\geq 0$. 
By adding this last inequality on the left hand side of inequality~(\ref{sum1}) above we get inequality~(\ref{Inxi2>x1x2}),   as desired.
For the last assertion of the lemma, as shown in the proof of Proposition~\ref{maximum} we have,
\begin{equation*}
    \dfrac{\partial^2 P_i}{\partial x_i^2}\dfrac{\partial^2 P_i}{\partial z_i^2} -
\biggl(  \dfrac{\partial^2 P_i}{\partial x_i\partial z_i} \biggr)^2 =- q_i (1-\sigma) F\sp{\prime\prime} \left(  x_i p(x_1+x_2)-z_1 \right)\dfrac{\partial^2}{\partial x_i^2}\left(x_ip(x_1+x_2)-\frac{C_i(x_i)}{1-\sigma}\right).
\end{equation*}
A similar computation shows
\begin{equation*}
    \frac{\partial^2P_i}{\partial x_1\partial x_2}\frac{\partial^2 P_i}{\partial z_i^2}- \frac{\partial^2 P_i}{\partial x_1\partial z_i}
    \frac{\partial^2 P_i}{\partial x_2\partial z_i}=-q_i(1-\sigma)F\sp{\prime\prime}(x_i p(x_1+x_2)-z_1)\frac{\partial^2}{\partial x_1\partial x_2} x_ip(x_1+x_2).
\end{equation*}
Thus, if $F\sp{\prime\prime}(x_i p(x_1+x_2)-z_i)>0$  then since also $q_i(1-\sigma)>0$ the desired equivalence is inmmediate from the last two equations.
\end{proof}

 The inequality $\dfrac{\partial^2 x_ip(x_1+x_2)}{\partial x_1 \partial x_2}<0$ in Lemma~\ref{lemma1} is referred as an strategic substitute condition  in \cite[Page~494]{Bulow}. This condition suggests that the optimal response of the $i$-firm in the duopoly is to be less aggressive when its competitor makes an aggressive play.

\begin{lemma}\label{lemma2}
  If $F\sp{\prime\prime}(x_i p(x_1+x_2)-z_i)>0$ and inequality  (\ref{In2}) hold then $-\dfrac{\partial^2 P_i}{\partial x_i^2}\geq - \dfrac{\partial^2 P_i}{\partial x_1\partial x_2} $
  If in addition we assume inequalities (\ref{Inxi2>x1x2}) and   $\left(\frac{\partial }{\partial x_1}+\frac{\partial}{\partial x_2}\right)x_ip(x_1+x_2)\geq 0$ then inequality (\ref{In3}) holds.
\end{lemma}

\begin{proof}
As was observed in the proof of Proposition~\ref{maximum}, inequality $F\sp{\prime\prime}(x_i p(x_1+x_2)-z_i)>0$ implies $\dfrac{\partial^2 P_i}{\partial z_i^2}<0$. On the other hand, we have
\begin{align*}
\left(\frac{\partial^2 P_i}{\partial x_i\partial z_i}  \right)^2 
    &= q_i^2 (F\sp{\prime\prime}(x_i p(x_1+x_2)-z_i))^2\left(\frac{\partial}{\partial x_i}x_ip(x_1+x_2)\right)^2 \\
    &=\frac{\partial^2 P_i}{\partial x_1 \partial z_i}\frac{\partial^2 P_i}{\partial x_2\partial z_i}+
    q_i^2F\sp{\prime\prime}(x_i p(x_1+x_2)-z_i)p(x_1+x_2)\left(\frac{\partial}{\partial x_i}x_ip(x_1+x_2)\right)\\
    &>\frac{\partial^2 P_i}{\partial x_1 \partial z_i}\frac{\partial^2 P_i}{\partial x_2\partial z_i}
\end{align*}
Using inequality~(\ref{In2}) and the last equality we then get
\begin{align*}
    \frac{\partial^2P_i}{\partial x_i^2}
    \frac{\partial^2 P_i}{\partial z_i^2}
    &> \frac{\partial^2P_i}{\partial x_1\partial x_2}
    \frac{\partial^2 P_i}{\partial z_i^2}
    -\frac{\partial^2 P_i}{\partial x_1 \partial z_i}\frac{\partial^2 P_i}{\partial x_2\partial z_i}
    +\left(\frac{\partial^2 P_i}{\partial x_i\partial z_i}\right)^2\\
    &> \frac{\partial^2P_i}{\partial x_1\partial x_2}
    \frac{\partial^2 P_i}{\partial z_i^2}
\end{align*}
Thus $-\dfrac{\partial^2P_i}{\partial x_i^2}>-\dfrac{\partial^2P_1}{\partial x_2\partial x_1}$.

Finally if we assume inequalities (\ref{Inxi2>x1x2}) and   $\left(\frac{\partial }{\partial x_1}+\frac{\partial}{\partial x_2}\right)x_ip(x_1+x_2)\geq 0$ we get
\begin{align*}
-\frac{\partial^2 P_i}{\partial x_i^2}- \frac{\partial^2 P_i}{\partial x_2\partial x_1}
    &=
    - (1-\sigma) \dfrac{\partial^2}{\partial x_i^2}\left(x_ip(x_1+x_2)-\frac{C_i(x_i)}{1-\sigma}
    \right)
    - (1-\sigma)\frac{\partial^2}{\partial x_1\partial x_2} x_ip(x_1+x_2)\\
    &\qquad
    + q_1\frac{C_i\sp\prime(x_i)}{1-\sigma}\left( \frac{\partial x_ip(x_1+x_2) }{\partial x_1}+\frac{\partial x_ip(x_1+x_2)}{\partial x_2}\right) F\sp{\prime\prime} \left(x_1p(x_1+x_2)-z_i\right)\\
    &>0.
\end{align*}
\end{proof}

We resume lemmas~\ref{lemma1} and \ref{lemma2} in the following.

\begin{corollary}\label{cor}
Suppose that for $i=1,2$ the following four inequalities hold. $F\sp{\prime\prime}(x_i p(x_1+x_2)-z_i)>0$ 
$C_i\sp{\prime\prime}(x_i)\geq 0$, $\dfrac{\partial^2 x_ip(x_1+x_2)}{\partial x_1 \partial x_2}\leq 0$,  and $\left(\frac{\partial }{\partial x_1}+\frac{\partial}{\partial x_2}\right)x_ip(x_1+x_2)\geq 0$. Then the inequalities (\ref{In2}) and (\ref{In3}) also hold.
\end{corollary}

Throughout the rest of the paper, we will assume inequalities  (\ref{In2}) and (\ref{In3}).

We now introduce the delay in Cournot duopoly model with tax evasion. We begin by assuming, for $i=1,2$, that $x_i(t)$ is the production of the $i$-firm at the time $t$. Similarly, $z_i(t)$ is the declared revenue of the $i$-firm at the time $t$. Let $\tau\geq 0$ denote the delay.  Define the new variables ${x_{i}}\sb\tau(t)= x_i(t-\tau)$ and ${z_{i}}\sb\tau(t)=z_i(t-\tau)$. 
 The profit function $P_1$  remains unchanged, as given in equation~(\ref{FuncionBeneficio}), since the delay does not affect the first firm.
%
%
For the case of profit function $P_2$, it is modified to reflect that the second firm enters the market after a delay $\tau\geq 0$, so that $x_1$ is replaced by ${x_{1}}\sb\tau$:
\begin{align}\label{P2Retardo}
P_{2} &  = \left(1 - q_2\right) \bigl(x_{2}\, p \left({x_1}\sb\tau+x_2\right) - C_{2} (x_{2}) - \sigma z_{2}\bigr)   
 \notag  \\ 
 &     \quad
+ \, q_2 \, \biggl(\left(1 - \sigma\right) x_{2}\, p \left({x_1}\sb\tau+x_2\right) - C_{2} (x_{2}) - F \bigl( x_{2}\, p \left({x_1}\sb\tau+x_2\right) - z_{2} \bigr)\biggr).   
\end{align}

We may regard $P_i$ as a function which depends on four variables  $\vec{x}=(x_1(t),x_2(t),z_1(t),z_2(t))$ and four delayed variables ${\vec{x}_{\tau}}=( {x_{1}}\sb\tau,{x_{2}}\sb\tau,{z_{1}}\sb\tau,{z_{2}}\sb\tau)$. 
  To describe $x_i\sp \prime(t)$ and $z_i\sp\prime(t)$, the variations of the production  and  the declared revenue of the competitors over time $t$, we will follow a gradient dynamics approach, that is to say, we assume that the firms adjust their outputs and their declared revenue proportionally to the rate of change of their profits. In other words:
\begin{equation*}
    \dfrac{dx_i}{dt}=k_i\dfrac{\partial P_i}{\partial x_i}\hspace*{1cm} \text{ and } \hspace*{1cm}
    \dfrac{dz_i}{dt}=k_{i+2}\dfrac{\partial P_i}{\partial z_i},
\end{equation*}
where $k_1,\ k_2,\ k_3,\ k_4>0$ are constants. The delay Cournot duopoly with tax evasion is then
\begin{align}\label{CournotRetardo}
   \dfrac{d\vec{x}}{dt}= 
\left( k_1\dfrac{\partial P_1}{ \partial x_1}(\vec{x},\vec{x}\sb\tau),
k_2\dfrac{\partial P_2}{ \partial x_2}(\vec{x},\vec{x}\sb\tau),
k_3\dfrac{\partial P_1}{ \partial z_1}(\vec{x},\vec{x}\sb\tau),
k_4\dfrac{\partial P_2}{ \partial z_2}(\vec{x},\vec{x}\sb\tau)
 \right).
 \end{align}
Notice that when $\tau=0$, the fixed point $(x_1\sp\ast,x_2\sp\ast,z_1\sp\ast,z_2\sp\ast)$ of the dynamical system (\ref{CournotRetardo})  is precisely  the  point computed in Proposition~\ref{maximum} where the profit functions $P_1$ and $P_2$ reach their maxima. In fact, notice that  the partial derivative of the function $P_2$ as defined in (\ref{FuncionBeneficio}) with respect to $x_1$  is equal to the the partial derivative of $P_2$, as defined in (\ref{P2Retardo}), with respect to ${x_1}\sb\tau$ when both are evaluated at the fixed point.

\begin{proposition}\label{formulas}
The quasipolynomial associated to the linearization of the delay Cournot duopoly model defined in (\ref{CournotRetardo}) at the fixed point $(x_1\sp\ast,x_2\sp\ast,z_1\sp\ast,z_2\sp\ast)$ which satisfy the system of equation (\ref{Cournot}) in Proposition~\ref{maximum} is given by the formula
\begin{equation}\label{PolCarac}
 Q\sb\tau(\lambda)=p_1\left(\lambda\right) p_2\left(\lambda\right) - e^{-\lambda\tau} g_1(\lambda)g_2(\lambda),
\end{equation}
\noindent
where 

$p_i(\lambda)= \lambda^2- \left(
  k_i \dfrac{\partial^2P_i}{\partial x_i^2}+     
  k_{i+2} \dfrac{\partial^2P_i}{\partial z_i^2}
  \right)\lambda +
  k_i \dfrac{\partial^2P_i}{\partial x_i^2}
  k_{i+2} \dfrac{\partial^2 P_i}{ \partial z_i^2}-
  k_i \dfrac{\partial^2P_i}{\partial x_i\partial z_i}
  k_{i+2} \dfrac{\partial^2 P_i}{\partial {x_i} \partial z_i}$

and

$g_i(\lambda)= k_{i} \dfrac{\partial^2P_i}{\partial x_2\partial x_1}\lambda -  k_i \dfrac{\partial^2P_i}{\partial x_2\partial x_1} k_{i+2} \dfrac{\partial^2 P_i}{\partial z_i^2} +  k_i\dfrac{\partial^2P_i}{\partial z_i\partial x_1}  k_{i+2}\dfrac{\partial^2 P_i}{\partial {x_2} \partial z_i}.$  

\end{proposition}
\begin{proof}
 The linearized system is of the form 
 \begin{equation}\label{lineal} 
 \frac{d\vec{x}}{dt}=A\vec{x}+B\vec{x}_{\tau},
 \end{equation}
 where 
{$$ 
\begin{array}{ccc}
A=\begin{pmatrix}
   k_1 \dfrac{\partial^2P_1}{\partial x_1^2} & k_1 \dfrac{\partial^2P_1}{\partial x_2\partial x_1}  &   k_1\dfrac{\partial^2P_1}{\partial z_1\partial x_1} & 0 \\
  0
     &k_2\dfrac{\partial^2 P_2}{\partial x_2^2}  & 0  &
     k_2 \dfrac{\partial^2 P_2}{\partial {z_2} \partial x_2} \\
    k_3 \dfrac{\partial^2 P_1}{\partial {x_1} \partial z_1} &
   k_3  \dfrac{\partial^2 P_1}{\partial {x_2} \partial z_1}  &
    k_3 \dfrac{\partial^2 P_1}{\partial z_1^2}         &   0\\
  0  &
 k_4 \dfrac{\partial^2 P_2}{\partial x_2\partial z_2} & 
   0&
   k_4\dfrac{\partial^2 P_2}{\partial z_2^2}
\end{pmatrix} & \text{and} & B=\begin{pmatrix}
   0 & 0 & 0 & 0 \\
   k_2 \dfrac{\partial^2 P_2}{\partial {x_1}\sb\tau \partial x_2}
     &0 & 0 & 0\\
    0 & 0  & 0 & 0\\
  k_4 \dfrac{\partial^2 P_2}{\partial {x_1}\sb\tau \partial z_2}   & 0&  0&  0
\end{pmatrix}
\end{array}. 
$$  }
The required quasipolynomial is nothing but $Q(\lambda)=\det(A+Be^{-\lambda\tau}-\lambda I)$. A standard computation verifies the proposition. We omit details.
\end{proof}

It will be useful to have conditions for the stability of the Cournot duopoly model when there is no delay, that is, for $\tau=0$.

\begin{proposition}\label{stable}
  The characteristic polynomial 
 \begin{align*}
       Q_0(\lambda)&=p_1(\lambda)p_2(\lambda)-g_1(\lambda)g_2(\lambda)\\
       &=\lambda^4+\alpha_3\lambda^3+\alpha_2\lambda^2+\alpha_1\lambda+\alpha_0,
\end{align*}
  corresponding to the linearization of the Cournot duopoly model~(\ref{CournotRetardo}) without delay is asymptotically stable if  
  $\alpha_1\alpha_2\alpha_3>\alpha_1^2+\alpha_3^2\alpha_0$.
\end{proposition}
\begin{proof}
We first observe that  the three inequalities $\alpha_0>0$, $\alpha_1>0$ and $\alpha_1\alpha_2\alpha_3>\alpha_1^2+\alpha_3^2\alpha_0$ imply $\alpha_2\alpha_3>\alpha_1+\dfrac{\alpha_3^2\alpha_0}{\alpha_1}>\alpha_1$. Thus, the four inequalities $\alpha_0>0$, $\alpha_1>0$, $\alpha_3>0$ and $\alpha_1\alpha_2\alpha_3>\alpha_1^2+\alpha_3^2\alpha_0$
imply $\dfrac{1}{\alpha_3}>0$, $\dfrac{\alpha_3^2}{\alpha_2\alpha_3-\alpha_1}>0$, $\dfrac{(\alpha_2\alpha_3-\alpha_1)^2}{\alpha_3(\alpha_1\alpha_2\alpha_3-\alpha_1^2-\alpha_3^2\alpha_0)}>0$ and $\alpha_0>0$, which are the Routh-Hurwitz conditions for the desired result. Thus, in order to establish our proposition, we only need to verify the three inequalities $\alpha_0>0$, $\alpha_1>0$  and $\alpha_3>0$. 

To prove $\alpha_3>0$, we notice $\alpha_3$ is  the sum of the linear coefficient of $p_1(\lambda)$ plus the linear coefficient of $p_2(\lambda)$, and as consequence of inequalities (\ref{In2}) and (\ref{In3}) we get
\begin{equation*}
    \alpha_3=-\sum_{i=1}^{2}\left(k_i \dfrac{\partial^2P_i}{\partial x_i^2}+     
  k_{i+2} \dfrac{\partial^2P_i}{\partial z_i^2}\right)>0.
\end{equation*}

On the other hand, using inequality~(\ref{In2}) we get

\begin{align*}
\prod_{i=1}^{2} \left(\frac{\partial^2P_i}{\partial x_i^2}\frac{\partial^2 P_i}{\partial z_i^2}- \left(\frac{\partial^2 P_i}{\partial x_i\partial z_i}\right)^2\right)
    &> \prod_{i=1}^{2}
    \left|\frac{\partial^2P_i}{\partial x_1\partial x_2}\frac{\partial^2 P_i}{\partial z_i^2}-\frac{\partial^2 P_i}{\partial x_1\partial z_i}
    \frac{\partial^2 P_i}{\partial x_2\partial z_i}\right| 
\end{align*}

And thus
    
\begin{align*}
    \alpha_0/(k_1k_2k_3k_4)&=\prod_{i=1}^{2}\left(\dfrac{\partial^2P_i}{\partial x_i^2}    
   \dfrac{\partial^2P_i}{\partial z_i^2}
  -\left(\dfrac{\partial^2P_i}{\partial z_i\partial x_i} \right)^2\right)
  - \prod_{i=1}^{2}
     \left( \dfrac{\partial^2P_i}{\partial x_{1}\partial x_2}  \dfrac{\partial^2 P_i}{\partial z_i^2} -
  \dfrac{\partial^2P_i}{\partial x_1\partial z_i} 
  \dfrac{\partial^2 P_i}{\partial {x_{2}} \partial z_i}   \right)\\
  &>0.
\end{align*}
Finally, for convenience, let us label $P_1(\lambda)=\lambda^2+a_1\lambda+a_0$, $P_2(\lambda)=\lambda^2+b_1\lambda+b_0$, $g_1(\lambda)=c_1\lambda+c_0$ and $g_2(\lambda)=d_1\lambda+d_0$. Using inequalities  (\ref{In2}) and (\ref{In3})  we obtain

%
\begin{equation*}
    b_1=
  -k_2 \dfrac{\partial^2P_2}{\partial x_2^2}-    k_4 \dfrac{\partial^2P_2}{\partial z_2^2}
> k_2 \left|\dfrac{\partial^2P_2}{\partial x_2^2}\right|\geq |d_1|
\end{equation*}
and
\begin{align*}
    a_0&=k_1k_3\frac{\partial^2P_1}{\partial x_1^2}\frac{\partial^2 P_1}{\partial z_1^2}-k_1k_3 \left(\frac{\partial^2 P_1}{\partial x_1\partial z_1}\right)^2\\
    &> \left|k_1k_3\frac{\partial^2P_1}{\partial x_1\partial x_2}\frac{\partial^2 P_1}{\partial z_1^2}-k_1k_3\frac{\partial^2 P_1}{\partial x_1\partial z_1}
    \frac{\partial^2 P_1}{\partial x_2\partial z_1}\right|\\
    &=|c_0|
\end{align*}

Thus $a_0b_1>|c_0d_1|\geq -c_0d_1$ and so $a_0b_1+c_0d_1>0$. Similarly, $a_1b_0+c_1d_0>0$. Thus $\alpha_1=a_0b_1+c_0d_1+a_1b_0+c_1d_0>0$, as wanted.

\end{proof}

\section{Same marginal production costs}

In this section we prove the main result of this paper. We will give conditions for the stability of the Cournot duopoly model with tax evasion defined in (\ref{CournotRetardo}) and in fact, we will also show that the stability won't be affected by any delay.

As mentioned in the introduction, we  first assume the marginal production costs and the second derivatives of the cost functions on the equilibrium point are equal, that is, 
\begin{equation}\label{samemc}
C\sp\prime_1(x_i\sp\ast)=C\sp\prime_2(x_2\sp\ast)\text{ and } C_1\sp{\prime\prime}(x_1^\ast)=C_i\sp{\prime\prime}(x_2^\ast).
\end{equation}
The following proposition is a slight generalization of an observation in \cite[Page~717]{GR}.

\begin{proposition}\label{symmetric}
Let $(x_1\sp\ast,x_2\sp\ast,z_1\sp\ast,z_2\sp\ast)$ be the equilibrium point of  the Cournot duopoly model defined in (\ref{CournotRetardo}). Then $C_1\sp\prime(x_1\sp\ast)=C_2\sp\prime(x_2\sp\ast)$ if and only if $x_1\sp\ast=x_2\sp\ast$. Furthermore,  $q_1=q_2$ if and only if $ z_2\sp\ast-z_1\sp\ast=\left(x_1\sp\ast-x_2\sp\ast\right) p\left(x_1\sp\ast+x_2\sp\ast \right)$. Thus, if $C_1\sp\prime(x_1\sp\ast)=C_2\sp\prime(x_2\sp\ast)$ and $q_1=q_2$ then $x_1\sp\ast=x_2\sp\ast$ and $z_1\sp\ast=z_2\sp\ast$
\end{proposition}
\begin{proof}
Subtracting the first equation in Proposition~\ref{maximum} for $i=1$ to $i=2$,  we obtain 
\begin{equation*}
    (x_1\sp\ast-x_2\sp\ast)p\sp\prime(x_1\sp\ast+x_2\sp\ast)=0.
\end{equation*}
Since $p$ is strictly decreasing we obtain $x_1\sp\ast=x_2\sp\ast$. Conversely, if $x\sp\ast=x_1\sp\ast=x_2\sp\ast$ then using again the first equation in Proposition~\ref{maximum}, we obtain $\frac{C_1\sp\prime(x_1^\ast)}{1-\sigma}=p(2x\sp\ast)+x\sp\ast p\sp\prime(2x\sp\ast)=\frac{C_2\sp\prime(x_2^\ast)}{1-\sigma}$ and so $C_1\sp\prime(x_1^\ast)=C_2\sp\prime(x_2^\ast)$. 

For the second assertion, assume $q_1=q=q_2$. Using the second equation in Proposition~\ref{maximum} we obtain 
\begin{equation*}
    F\sp\prime(x_1\sp\ast p(x_1\sp\ast+x_2\sp\ast)-z_1\sp\ast)=
    \frac{(1-q_1)\sigma}{q_1}=\frac{(1-q)\sigma}{q}=\frac{(1-q_2)\sigma}{q_2}=F\sp\prime(x_2\sp\ast p(x_1\sp\ast+x_2\sp\ast)-z_2\sp\ast)
\end{equation*}
and since $F\sp\prime$ {is strictly increasing} then it is one-to-one, so we obtain $x_1\sp\ast p(x_1\sp\ast+x_2\sp\ast)-z_1\sp\ast=x_2\sp\ast p(x_1\sp\ast+x_2\sp\ast)-z_2\sp\ast$ from where the desired conclusion follows. Conversely, if $ z_2\sp\ast-z_1\sp\ast=\left(x_1\sp\ast-x_2\sp\ast\right) p\left(x_1\sp\ast+x_2\sp\ast \right)$ then using again the second equation in Proposition~\ref{maximum} we get
\begin{equation*}
    \frac{(1-q_1)\sigma}{q_1}=
    F\sp\prime(x\sp\ast p(2x\sp\ast)-z\sp\ast)=
    \frac{(1-q_2)\sigma}{q_2},
\end{equation*}
which implies $q_1=q_2$, as wanted. The last assertion of the proposition follows from the previous ones.
    \end{proof}

In addition to the hypothesis on the marginal costs, we are going to assume the the probabilities $q=q_1=q_2$ of being caught evading taxes are the same for both firms. We remark that this assumption $q_1=q_2$ was also made in the original introduction of the Cournot duopoly model with tax evasion \cite{GR} and also in subsequent works, e.g~\cite{N}. In addition, we will assume $k_1=k_2$ and $k_3=k_4$, that is, both firms have the same strategies for adapting their productions and income declared for tax purposes. 
For easy reference, we establish the following.

\begin{definition}\label{definition_NLCournott}
We define the {\em Cournot duopoly model with equal marginal costs}, as  the delay Cournot duopoly model with tax evasion (\ref{CournotRetardo}) such that conditions (\ref{samemc}) is satisfied together with the equalities $q_1=q_2$, $k_1=k_2$ and $k_3=k_4$.
\end{definition}

In the next proposition we will be able to establish that under the equal marginal cost condition given in Definition~\ref{definition_NLCournott} a simplification of the quasipolynomial (\ref{PolCarac}) is possible.
\begin{proposition}\label{quasiequalcost}
For the  Cournot duopoly model  with equal marginal costs given in Definition~\ref{definition_NLCournott}, the second derivatives of $P_1$ and $P_2$ satisfy the following equalities evaluated at the equilibrium point

$\dfrac{\partial^2P_1}{\partial z_1\partial x_1}=\dfrac{\partial^2P_2}{\partial z_2\partial x_2}$, $\dfrac{\partial^2P_1}{\partial z_1^2}=\dfrac{\partial^2P_2}{\partial z_2^2}$, $\dfrac{\partial^2P_1}{\partial x_1^2}=\dfrac{\partial^2P_2}{\partial x_2^2}$, $\dfrac{\partial^2P_1}{\partial x_2\partial z_1}=\dfrac{\partial^2P_2}{\partial {x_1}\sb\tau \partial z_2}$ and $\dfrac{\partial^2P_1}{\partial x_2\partial x_1}=\dfrac{\partial^2P_2}{\partial {x_1}\sb\tau\partial x_2}$

 Thus,  the quasipolynomial (\ref{PolCarac}) corresponding to the linearization of the form~(\ref{lineal})  reduces to
  $Q(\lambda)=p^2(\lambda)-e^{-\tau\lambda}g^2(\lambda)$, where $p(\lambda)$ and $g(\lambda)$ are second and first degree polynomials, respectively.
\end{proposition}
\begin{proof}
   Using Proposition~\ref{formulas} and Proposition~\ref{symmetric}, the result is a direct computation. 
\end{proof}

The following theorem formalizes the claim that the conditions for the Cournot duopoly model with equal marginal costs given in Definition~\ref{definition_NLCournott}, are sufficient for the equilibrium point to be asymptotically stable.  In particular, they must satisfy the equivalent conditions for the model presented in \cite[Proposition~5]{N}.

\begin{theorem}\label{AS}
The equilibrium point of the Cournot duopoly model with equal marginal costs given in Definition~\ref{definition_NLCournott} is asymptotically stable for $\tau= 0$.
\end{theorem}

\begin{proof}
Using Proposition~\ref{quasiequalcost}, we may assume $p(\lambda)=\lambda^2+a_1\lambda+a_0$ and $g(\lambda)=c_1\lambda+c_0$, so that the characteristic polynomial corresponding to the linearization of the Cournot duopoly model with equal marginal costs and without delay, is
\begin{align*}
    Q(\lambda)&=(\lambda^2+a_1\lambda+a_0)^{2}-(c_1\lambda+c_0)^2\\
    &= \lambda^4+2a_1\lambda^3+(a_1^2+2a_0-c_1^2)\lambda^2+2(a_0a_1-c_0c_1)\lambda+a_0^2-c_0^2\\
    &=\lambda^2+\alpha_3\lambda^3+\alpha_2\lambda^2+\alpha_1\lambda+\alpha_0.
\end{align*}
According to Proposition~\ref{stable}, we only need to verify that  $\alpha_1\alpha_2\alpha_3>\alpha_1^2+\alpha_3^2\alpha_0$.

%
%
%
Using inequality~(\ref{In3}) we have $a_1>|c_1|$
then $a_1^2-c_1^2>0.$ Furthermore, we know from Proposition~\ref{stable} that $\alpha_1>0$ and $\alpha_3>0$. Thus, we compute
\begin{align*}
    \alpha_3\alpha_2\alpha_1 &= 
    (2a_1)(2a_0+a_1^2-c_1^2)(2a_0a_1-2c_0c_1)\\
    &=(2a_1)(2a_0)(2a_0a_1-2c_0c_1)+(2a_1)(a_1^2-c_1^2)(2a_0a_1-2c_0c_1)\\
    &=4a_0^2a_1^2-8a_0a_1c_0c_1+4c_0^2c_1^2
    +4a_0^2a_1^2-4c_0^2c_1^2+(2a_1)(a_1^2-c_1^2)(2a_0a_1-2c_0c_1)\\
    &=\alpha_1^2 + 4a_1^2(a_2^2-c_0^2) +4a_1^2c_0^2-4c_0^2c_1^2+(2a_1)(a_1^2-c_1^2)(2a_0a_1-2c_0c_1)\\
    &=\alpha_1^2+\alpha_3^2\alpha_0+(a_1^2-c_1^2)\left(4c_0^2+ \alpha_3(a_1^2-c_1^2)\alpha_1\right)\\
    &>\alpha_1^2+\alpha_3^2\alpha_0,
\end{align*}
as was to be proved.
\end{proof}

\begin{proposition}\label{noimaginaryroots}
  {The quasipolynomial (\ref{PolCarac})} corresponding to the linearization~(\ref{lineal}) of the  Cournot duopoly model with equal marginal costs given in Definition~\ref{definition_NLCournott} does not have purely imaginary roots. 
\end{proposition}
\begin{proof}
By contradiction, suppose $\lambda=i\omega$, $\omega\not=0$, is an imaginary root of $Q\sb\tau(\lambda)=p^2(\lambda)-e^{i\omega\tau}g^2(\lambda)$. Therefore $|p(i\omega)|^2=|g(i\omega)|^2$. 
On the other hand,

Hence
\begin{align*}
    |p(i\omega)|^2&=\mathrm{Re}^2p(iw)+ \mathrm{Im}^2p(iw)\\
    &=\left( (k_1k_3)\left(\frac{\partial^2P_1}{\partial x_i^2}\frac{\partial^2 P_1}{\partial z_1^2}-\left(\frac{\partial^2 P_1}{\partial x_1\partial z_1}\right)^2\right)-\omega^2\right)^2\\
    &\qquad + \omega^2 \left( k_1 \frac{\partial^2 P_1}{\partial x_1^2}+k_3\frac{\partial^2 P_1}{\partial z_1^2} \right)^2 \\
    &= w^4 + w^2\left( k_1^2\left(\frac{\partial^2 P_1}{\partial x_1^2}\right)^2 
    + k_3^2\left(\frac{\partial^2 P_1}{\partial z_1^2}\right)^2
    +2k_1k_3\left(\frac{\partial^2 P_1}{\partial x_1\partial z_1}\right)^2\right)\\
    &\qquad + (k_1k_3)^2
    \left(\frac{\partial^2P_1}{\partial x_i^2}\frac{\partial^2 P_1}{\partial z_1^2}-\left(\frac{\partial^2 P_1}{\partial x_1\partial z_1}\right)^2\right)^2\\
    &>  w^4 + w^2 k_1^2\left(\frac{\partial^2 P_1}{\partial x_1^2}\right)^2+ (k_1k_3)^2 \left(\frac{\partial^2P_1}{\partial x_1\partial x_2}\frac{\partial^2 P_1}{\partial z_1^2}-\frac{\partial^2 P_1}{\partial x_1\partial z_1}
    \frac{\partial^2 P_1}{\partial x_2\partial z_1}\right)^2 \\
      &=\mathrm{Re}^2g(i\omega) 
                  +\mathrm{Im}^2 g(i\omega)\\
                  &=\left|g(i\omega)\right|^2.
\end{align*}
Hence $|p(i\omega)|>|g(i\omega)|$,
a contradiction. This completes the proof.
\end{proof}
We have proved in Theorem~\ref{AS} that the equilibrium point of the Cournot duopoly model with equal marginal costs is asymptotically stable for $\tau=0$, and in the previous Proposition~\ref{noimaginaryroots} we showed that the quasipolynomial corresponding to its linearization does not have imaginary roots. Therefore, there are no roots of the quasipolynomial which cross the imaginary axis as the value of the delay $\tau$ increases. As consequence, see e.g.~\cite{Cv}, we have obtained the following theorem, the main result of our paper.
\begin{theorem}\label{theorem}
  The equilibrium point of the Cournot duopoly model with equal marginal costs given in Definition~\ref{definition_NLCournott} is asymptotically stable and  independent of the delay.
\end{theorem}

{\begin{remark}
As a direct consequence of the above theorem, we can assure that the Cournot duopoly model with equal marginal costs does not have bifurcations under parametric variations of the delay $\tau\geq0$.
\end{remark}}

\section{Examples}\label{examples}
In this section we will provide some examples of functions which satisfy the conditions for the Cournot duopoly model with equal marginal costs given in Definition~\ref{definition_NLCournott}. As motivation, let us observe that by Proposition~\ref{symmetric}, the system of equations~(\ref{Cournot}) can be rewritten as
\begin{align}\label{eq1}
    p(2x\sp\ast)+x\sp\ast p\sp\prime(2x\sp\ast)&=\frac{C_1\sp\prime(x\sp\ast)}{1-\sigma}\\  \label{eq2}
    F\sp\prime(x\sp\ast p(2x\sp\ast)-z\sp\ast)&=\frac{\sigma(1-q)}{q} 
\end{align}
Let $G=F\sp\prime$ then both $G$ and $G^{-1}$  are strictly increasing by the definition of $F$, so that solving for $z\sp\ast$ in Equation~(\ref{eq2}) we obtain
\begin{equation*}
      z\sp\ast=x\sp\ast p(2x\sp\ast) -G^{-1}\left(\frac{\sigma(1-q)}{q}\right).
\end{equation*}
Thus, the amount $z\sp\ast$ declared as revenue by the firms in the duopoly will be closer to their actual revenue $x\sp\ast p(2x\sp\ast)$  either when  the effectiveness of audits is increased, that is,  the value of $q$ representing the probability of being caught evading taxes increases, or
by adjusting the penalties for tax evasion, that is, introducing a penalty function such that the value of $G^{-1}\left(\dfrac{\sigma(1-q)}{q}\right)$ is as low as possible. In addition, Theorem~\ref{theorem} assures that the equilibrium point of our model is asymptotically stable and independent of time delay.

We will next provide a family of cost functions which satisfy condition (\ref{samemc}) and two examples of inverse demand functions which further fulfill the hypotheses of Corollary~\ref{cor}. Thus, for our examples, the choice of penalty function $F$ does not affect the stability of the system, we only need to verify that $F$ satisfy $F\sp{\prime\prime}(x_i p(x_1+x_2)-z_i)>0$, say $F(x)=2x^2$. 

\subsection{Examples of cost functions}
In this subsection we provide classes of cost functions which satisfy the required condition (\ref{samemc}). For $i=1,2$, define $C_i(x_i)=f_i+dx_i+cx_i^2$, where $f_i\geq 0$, $d>0$ and $c\geq 0$ are constants. We claim that this functions satisfy condition (\ref{samemc}). Indeed, subtracting the first equation of the system (\ref{Cournot}) for $i=2$ from $i=1$ we obtain
\begin{equation*}
   (1-\sigma)(x_1\sp\ast-x_2\sp\ast)p\sp\prime(x_1\sp\ast+x_2\sp\ast)=2c(x_1\sp\ast-x_2\sp\ast). 
\end{equation*}
By contradiction, if $x_1\sp\ast\not= x_2\sp\ast$ then the above equation implies $p\sp\prime(x_1\sp\ast+x_2\sp\ast)=\frac{2c}{1-\sigma}\geq 0$, contradicting that $p\sp\prime(x)<0$. An application of Proposition~\ref{symmetric} completes the proof of the claim. 

\subsection{Examples of inverse demand functions}
We now provide two examples of inverse demand functions $p(x)$ and analyze the conditions needed to  satisfy the inequalities $p\sp\prime(2x\sp\ast)+x\sp\ast p\sp{\prime\prime}(2x\sp\ast)=\dfrac{\partial^2 x_ip(x_1+x_2)}{\partial x_1 \partial x_2}\leq 0$,  and $\left(\frac{\partial }{\partial x_1}+\frac{\partial}{\partial x_2}\right)x_ip(x_1+x_2)=p(2x\sp\ast)+2x\sp\ast p\sp\prime(2x\sp\ast)\geq 0$
which, according to Corollary~\ref{cor}, are conditions for the required  inequalities.

First, consider $p(x)=1/x$.
We compute
$ 
p(2x\sp\ast)+2x\sp\ast p\sp\prime(2x\sp\ast)=
\frac{1}{2x\sp\ast}-\frac{2x\sp\ast}{4(x\sp\ast)^2}
=0,
$ 
and thus $p(2x\sp\ast)+2x\sp\ast p\sp\prime(2x\sp\ast)\geq 0$, as desired.  Furthermore
$
p\sp\prime(2x\sp\ast)+x\sp\ast p\sp{\prime\prime}(2x\sp\ast)
=-\frac{1}{4(x\sp\ast)^2}+x\sp\ast\frac{2}{8(x\sp\ast)^3}\leq 0,
$ 
as wanted. We conclude that the cost functions $C_i(x_i)=f_i+dx_i+cx_i^2$ from the previous subsection and the inverse demand function $p(x)=\frac{1}{x}$ always satisfy the  conditions from Corollary~(\ref{cor}). Hence,  when $C_i(x)=f_i+dx_i+cx_i^2$, $p(x)=\frac{1}{x}$, $q_1=q_2$, $k_1=k_3$ and $k_2=k_4$, then the Cournot duopoly model (\ref{CournotRetardo}) satisfy the conditions of Definition~\ref{definition_NLCournott}  and so, by Theorem~\ref{theorem}, is asymptotically stable and  independent of the delay. 

Finally, we consider a second inverse demand function $p(x)=a-bx$ with $a,b>0$.  We use Equation~(\ref{eq1}) to obtain $x\sp\ast=\dfrac{a}{3b}-\dfrac{C_1\sp\prime(x\sp\ast)}{3b(1-\sigma)}$.  Using the cost functions $C_i(x_i)=f_i+dx_i+cx_i^2$ from the previous subsection, we solve $x\sp\ast=\dfrac{a(1-\sigma)-d}{3b(1-\sigma)+2c}$. Clearly the inequality $p\sp\prime(2x\sp\ast)+x\sp\ast p\sp{\prime\prime}(2x\sp\ast)=-b\leq 0$ holds. Since the inequality $p(2x\sp\ast)+2x\sp\ast p\sp\prime(2x\sp\ast)\geq 0$ is equivalent to the inequality
$a-4bx\sp\ast\geq 0$,  we substitute the value of $x\sp\ast$ and obtain
\begin{equation}\label{condition4stab}
  2ac+ 4bd\geq ab\left(1-\sigma\right),
\end{equation}
which is the desired condition for stability. Hence, when $C_i(x)=f_i+dx_i+cx_i^2$, $p(x)=a-bx$, $q_1=q_2$, $k_1=k_3$, $k_2=k_4$ and inequality $(\ref{condition4stab})$ holds, then the Cournot duopoly model (\ref{CournotRetardo}) satisfy the conditions of Definition~\ref{definition_NLCournott} and those of Corollary~\ref{cor} so, by Theorem~\ref{theorem}, is asymptotically stable and  independent of the delay. 
When inequality~(\ref{condition4stab}) is not satisfied then instability may or may not occur: for example, in case $C_i(x_i)=4x_i$ and $p(x)=a-bx$, in other words, we fix the values  $c=0=f_i$ and $d=4$, then inequality~(\ref{condition4stab}) becomes $\frac{160}{9}\geq a$; notice the inequality does not depend on $b$. In case $a=80$, $b=10$,  $\sigma=0.1$, $q_i=0.5$, $k_i=1$, and $F(x)=2 x^2$  then
the equilibrium point satisfies $x_i\sp\ast=2.518518519$, $z_i\sp\ast=74.59777092$ and will be unstable for all $\tau\geq 0$. In Figure~\ref{graf} we show the roots of the corresponding quasipolynomial for several values of $\tau$. In fact, computer simulations show that there is a bifurcation for some value $b=b_0$ with $68<b_0<69$, more precisely, the system will be stable for $b>b_0$ and unstable for $b<b_0$.

\begin{figure}[ht]
\centering
\includegraphics[scale=0.4]{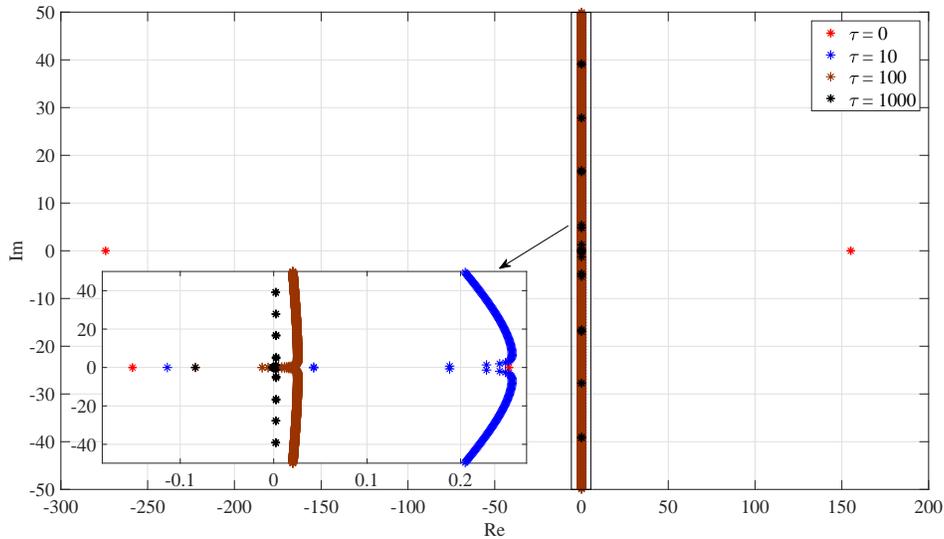}
\caption{Example of an unstable Cournot duopoly model.}
\label{graf}
\end{figure}

\section{Conclusions}

In this paper we presented a stability analysis for a Cournot duopoly model with tax evasion under parametric variations of the delay. As  a consequence of this analysis, we are able to provide conditions for the asymptotic stability of the equilibrium point. It is further proved that these same conditions imply the independence of the stability  under parametric variations of the time delay. In particular, there will be no bifurcations under parametric variations of the time delay.  
Our conditions for stability are surprisingly simple and apply for a variety of classical functions found in the literature, as was exhibited in previous section. The examples provided also show that our conditions are not necessary: the system may be stable despite not satisfying the conditions.

Under the proposed assumptions, we are able to suggest that either by increasing the effectiveness of audits or by adapting the penalties for tax evasion it may result in  the rise of the tax revenue, more precisely, the rise of the declared amount of revenue to the tax authority,   thus  inhibiting tax evasion and increasing public revenue. Furthermore, under the given conditions, the equilibrium point of the duopoly is not made unstable by a variation of the delay in the insertion of the second firm of the duopoly in the market.

\end{document}